\documentclass{article}%
\usepackage{amssymb}
\usepackage{amsfonts}
\usepackage{amsmath}
\usepackage{graphicx}
\usepackage{amssymb,color}%
\providecommand{\U}[1]{\protect\rule{.1in}{.1in}}

\definecolor{c20}{rgb}{0.,0.0,0.}
\definecolor{c30}{rgb}{0.,0.,0.}
\definecolor{c40}{rgb}{0,0.0,0.0}
\definecolor{c50}{rgb}{0,0,0}

\newtheorem{theorem}{Theorem}

\newtheorem{conclusion}{Conclusion}

\newtheorem{corollary}{Corollary}

\newtheorem{example}{Example}

\newtheorem{proposition}{Proposition}
\newtheorem{remark}{Remark}

\newenvironment{proof}[1][Proof]{\noindent\textbf{#1.} }{\ \rule{0.5em}{0.5em}}
\begin{document}

\title{On the new properties of conditional expectations with applications in finance }
\author{Ismihan Bayramoglu\\Department of Mathematics, Izmir University of Economics\\Izmir, Turkey}
\maketitle
\maketitle

\begin{abstract}
The concept of conditional expectation is important in applications of
probability and statistics in many areas such as reliability engineering,
economy, finance, and actuarial sciences due to its property of being the best
predictor of a random variable as a function of another random variable. This
concept also is essential in the martingale theory and theory of Markov
processes. Even though, there has been studied and published many interesting
properties of conditional expectations with respect to a sigma-algebra
generated by a random variable it remains an attractive subject having
interesting applications in many fields. In this paper, we present some new
properties of the conditional expectation of a random variable given another
random variable and describe useful applications in problems of
per-share-price of stock markets. The copula and dependence properties of
conditional expectations as random variables are also studied. We present also
some new equalities having interesting applications and results in martingale
theory and Markov processes.


\textbf{Conflicts of interest statement: }We\textbf{ } declare that have no
conflicts of interest.

\end{abstract}

\section{Introduction}

Let ($\Omega,\digamma,P)$ be a probability space. Consider the random
variables $X$ and $Y$ defined in this probability space and having joint
distribution function $F_{X,Y}(x,y)=C(F_{X}(x),F_{Y}(y)),(x,y)\in%
\mathbb{R}
^{2},$ where $C(u,v),(u,v)\in I^{2}\equiv\lbrack0,1]^{2}$ is a connected
copula. Consider functions $\varphi(y)=E(X\mid Y=y),y\in%
\mathbb{R}
$ and $\psi(x)=$ $E(Y\mid X=x),x\in%
\mathbb{R}
$ and random variables $Z_{1}\equiv\varphi(Y)=E(X\mid Y)$ and $Z_{2}\equiv
\psi(X)=E(Y\mid X).$ It is well known that the best predictor of $X$ by $Y$ in
the sense of least square \ distance is $\varphi(Y),$ and the best predictor
of $Y$ by $X$ is $\psi(X),$ i.e.
\begin{align*}
\min_{g}E(X-g(Y))^{2}  &  =E(X-\varphi(Y))^{2}\\
\min_{h}E(Y-h(X))^{2}  &  =E(Y-\psi(X))^{2},
\end{align*}
where the min is taken over all measurable functions $g$ and $h.$ The
conditional expectation $E(X\mid Y)$ is a random variable defined in the same
probability space ($\Omega,\digamma,P)$. \ The random variable defined as a
conditional expectation $E(X\mid Y)$ is an important classical concept, it is
the best predictor for $X$ \ as a function of $Y,$ and plays a crucial role in
many theoretical and practical aspects of probability theory. For example, in
practical applications, if we know the joint distribu\i tion of $X$ and $Y$
and the value of $Y,$ we can use $\varphi(Y)$ instead of random variable $X,$
whose values are very difficult, expensive, or impossible to measure. A wide
description of the concept of conditional expectation and its properties can
be found in many books on probability and statistics including Ross (2002),
Carlin and Taylor (1975), and Borovkov (1998) among others. \ In this paper,
we aim to consider some unknown and interesting properties of conditional
expectations having applications in many areas such as economics, engineering,
actuarial sciences, and financial mathematics. \ 

The paper is organized as follows. In chapter 1 we consider the conditional
expectation of the random variable $Y$ given $X$ and compare it with the
random variable defined as the arithmetic mean\ of conditional expectations of
$Y$ given $X_{1},X_{2},...,X_{n}$ which are the copies (dependent or
independent) \ of $X.$ $\ $\ The application of the results in finance is
shown. In Chapter 2 we are interested in the joint distribution of random
variables $\varphi(Y)$ and $\psi(X)$ and study the dependence properties and
copulas of these random variables. In Section 3 we consider the sequence of
any random variables $X_{1},X_{2},...,X_{n},...$ and study the properties of
the sequence of random variables defined as $\ Y_{1}=X_{1},Y_{2}=E(X_{2}\mid
X_{1}),...,Y_{n}=E(X_{n}\mid X_{n-1}),...$ and the sequence of random
variables defined as $E(X_{n+1}\mid X_{i_{1}},X_{i_{2}},...,X_{i_{k}}),1\leq
i_{1}<i_{2}<...<i_{k}\leq n,1\leq k\leq n.$ $\ $\ We present some theorems
describing the interesting properties of these sequences and provide examples
comparing them with Markov sequences and martingales.

\section{Inequalities with application to per-share-price of stock}

Let $X_{1},X_{2},...,X_{n}$ be the copies of the random variable $X.$ Let
\[
\hat{Y}=E(Y\mid X)
\]%
\[
\bar{Y}=\frac{1}{n}\sum\limits_{i=1}^{n}E(Y\mid X_{i}).
\]
The following theorem compares the predicted value of $Y$ through $\hat{Y}$
\ with the predicted value of $Y$ through $\bar{Y}.$

\begin{theorem}
\label{Theorem 4AA copy(1)}It is true that%
\begin{align*}
E(Y-\bar{Y})^{2}  &  \leq E(Y-\hat{Y})^{2},\\
E\left(  Y-\frac{1}{n}\sum\limits_{i=1}^{n}E(Y\mid X_{i})\right)  ^{2}  &
\leq E\left(  Y-E(Y\mid X)\right)  ^{2}.
\end{align*}

\end{theorem}

\begin{proof}
Using Shwarz inequality we can write%
\[
E\left(  Y-\frac{1}{n}\sum\limits_{i=1}^{n}E(Y\mid X_{i})\right)
^{2}=E\left(  \frac{1}{n}(nY-\sum\limits_{i=1}^{n}E(Y\mid X_{i}))\right)  ^{2}%
\]%
\[
=\frac{1}{n^{2}}E\left(  \sum\limits_{i=1}^{n}(Y-E(Y\mid X_{i}))\right)
^{2}=\frac{1}{n^{2}}\left\{  E\left(  \sum\limits_{i=1}^{n}\left(  Y-E(Y\mid
X_{i})\right)  ^{2}\right)  \right.
\]%
\begin{align*}
&  \left.  +2\sum\limits_{1\leq i<j\leq n}E\left[  Y-E(Y\mid X_{i})\left(
Y-E(Y\mid X_{j})\right)  \right]  \right\} \\
&  \leq\frac{1}{n^{2}}E\left(  \sum\limits_{i=1}^{n}\left(  Y-E(Y\mid
X_{i})\right)  ^{2}\right) \\
&  +2\sum\limits_{1\leq i<j\leq n}\left[  E\left(  Y-E(Y\mid X_{i})\right)
^{2}\right]  ^{\frac{1}{2}}\left[  E\left(  Y-E(Y\mid X_{j})\right)
^{2}\right]  ^{\frac{1}{2}}%
\end{align*}%
\begin{align*}
&  =\frac{1}{n^{2}}\left(  nE\left(  Y-E(Y\mid X)\right)  ^{2}\right)
+2\frac{n(n-1)}{2}E\left(  Y-E(Y\mid X)\right)  ^{2}\\
&  =\frac{1}{n^{2}}\left(  (n+n^{2}-n)E(Y-E(Y\mid X))^{2}\right)  =E(Y-E(Y\mid
X))^{2}.
\end{align*}
Thus the theorem proved.
\end{proof}

\bigskip

This theorem may also be formulated as an important consequence of the
following general theorem.

\begin{theorem}
\label{Theorem 5AA}Let $X$ any $Y$ be any random variables defined on the same
propbability space and $X_{1},X_{2},...,X_{n}$ be the copies of $X,$ i.e
\ random variables (dependent or independent) having the same distribution as
$X.$ Then
\[
E\left(  Y-\frac{1}{n}\sum\limits_{i=1}^{n}X_{i}\right)  ^{2}\leq E(Y-X)^{2}.
\]

\end{theorem}

\begin{proof}%
\begin{align*}
&  E\left(  Y-\frac{1}{n}\sum\limits_{i=1}^{n}X_{i}\right)  ^{2}\\
&  =\frac{1}{n^{2}}E\left(  nY-\sum\limits_{i=1}^{n}X_{i}\right)  ^{2}\\
&  =\frac{1}{n^{2}}E\left(  \sum\limits_{i=1}^{n}(Y-X_{i})\right)  ^{2}\\
&  =\frac{1}{n^{2}}\left(  \sum\limits_{i=1}^{n}E(Y-X_{i})^{2}+2\sum
\limits_{1\leq i<j\leq n}^{n}E(Y-X_{i})E(Y-X_{j})\right) \\
&  \leq\frac{1}{n^{2}}\left(  \sum\limits_{i=1}^{n}E(Y-X_{i})^{2}%
+2\sum\limits_{1\leq i<j\leq n}^{n}\left(  E(Y-X_{i})^{2}\right)  ^{\frac
{1}{2}}\left(  E(Y-X_{j})^{2}\right)  ^{\frac{1}{2}}\right) \\
&  =\frac{1}{n^{2}}\left(  nE(Y-X)^{2}+\frac{n(n-1)}{2}2E(Y-X)^{2}\right) \\
&  =\frac{1}{n^{2}}\left(  (n+n^{2}-n)E(Y-X))^{2}\right)  =E(Y-X)^{2}.
\end{align*}

\end{proof}

It is clear that if $E(Y$ $\mid$ $X)$ is used instead of $X$ in Theorem 2,
then Theorem 1 is obtained.

The simple Theorem \ref{Theorem 4AA copy(1)} may have interesting and
important applications. Below we provide an example demonstrating an
application of this theorem in the activity of brokers coalition in stock markets.

\begin{example}
(Brokers in coalition) Assume that $n$ brokers work in the financial market.
Each broker has a goal, to buy a fixed number of shares of some popular stock.
Let $\ X_{i}$ be the suggested price per - share to the market by the $i$th
broker respectively, $i=1,2,...,n,$ i.e. $\ i$th broker suggests that wants to
buy this share at a price $X_{i}.$ Obviously, the traded price in the market
is then expected to be $\max(X_{1},X_{2},...,X_{n})=X_{n:n},$ where
$X_{i:n},1\leq i\leq n$ is the $i$th order statistic of $X_{1},X_{2}%
,...,X_{n}.$ Assume that the $i$th broker wants to have a piece of information
about whether he wins or loose the came, for a given (suggested by himself)
price per share $X_{i}=x.$ Therefore $P(X_{i}>\max(X_{1},X_{2},...,X_{i-1}%
,X_{i+1},...,X_{i})$ is the probability that $i$th broker will certainly have
the stock, because he/she suggested more than others. Let $\psi_{i}(x)=$
$E(X_{n:n}\mid X_{i}=x).$ Now, it is known that the conditional r.v.
$\ \psi_{i}(X_{i})=E(X_{n:n}\mid X_{i})$ which gives information about the
traded-per-share given a suggested price by broker \ is the predictor for
$X_{n:n}$ by $X_{i}$ in the sense of mean squared error. Then the value
$X_{n:n}\mid X_{i}=x$ may be used as a measure of a successful trading
strategy for $i$th broker. Similarly, one can predict the predicted price per
- share $X_{i},i=1,2,...,n,$ given the traded price in the market, which is
$E(X_{i}\mid X_{n:n}=y),$ $i=1,2,...,n$ is the best approximate by the $i$th
broker given the traded price $y.$ Note that the historical values of
$X_{n:n}$ are known and can be provided from the marked. It is clear that
\ if
\[
E(X_{n:n}-E(X_{n:n}\mid X_{1}))^{2}\leq E(X_{n:n}-E(X_{n:n}\mid X_{2}))^{2},
\]
then $X_{1}$ is a better price than $X_{2}$ per-share price of a stock. Now
assume that $n$ brokers are working in the market as partners in coalition and
$X_{1},X_{2},...,X_{n}$ are the suggested prices of these brokers for a stock
share put on the market, respectively. Denote by $Y$ $\ $the max price
suggested by non-coalition brokers. Obviously, $Z=\max(X_{1},X_{2}%
,...,X_{n},Y)$ is the traded price in the market. Then $E(Z\mid X_{i})$ is the
information trade price given $X_{i}$ and $E(Z\mid Y)$ is the information
trade price per-share given by non-coalition brokers. Since,
\[
\frac{1}{n}\sum\limits_{i=1}^{n}E(Z\mid X_{i})
\]
is the information trade price per-share average of all brokers in coalition
by Theorem 1, we have
\begin{align*}
E\left(  Z-\frac{1}{n}\sum\limits_{i=1}^{n}E(Z\mid X_{i})\right)  ^{2}  &
\leq E(Z-E(Z\mid X_{i}))^{2},\\
i  &  =1,2,...,n.
\end{align*}
i.e. the predictor for a price of the share given by any of the member in the
coalition is worse than the average value of predictors of all brokers in the
coalition. This means that if a group of brokers is working in coalition they
will have better gain than if they work in solitary.
\end{example}

\bigskip

\section{Copula and coovariance}

Since $E(X\mid Y)$ is the best predictor for $X$ as a function of $Y,$ and
$E(Y\mid X)$ is the best predictor of $Y$ as a function of $X,$ it would be
interesting to investigate how the dependence structure will change if we
replaced $X$ with $E(X\mid Y)$ and $Y$ with $E(Y\mid X).$ For this purpose
consider the joint distribution of the random variables $Z_{1}$ $\equiv
\varphi(Y)=E(X\mid Y)$ and $Z_{2}\equiv\psi(X)=E(Y\mid X).$ \ We are
interested in copula of $Z_{1}$\ and $Z_{2}.$ Let $\varphi^{-1}(y)=\inf
\{x:\varphi(x)\leq y\}$ and $\psi^{-1}(x)=\inf\{y:\psi(y)\leq x\}$ \ are the
generalized inverses of $\varphi$ and $\psi.$ Consider the joint distribution
function of $Z_{1}$ and $Z_{2}.$ Let $F_{X}(x)$ and $F_{Y}(x)$ be a
distribution function of $X$ and $Y,$ respectively. Assuming $X$ and $Y$ have
the same support, denote left and right endpoints of the support of $X$ and
$Y$ by $a=\inf\{x:F_{X}(x)>0\}$ and $b=\sup\{x:F_{X}(x)<1\},$ respectively.
We\ allow also the cases $a=-\infty$ and $b=\infty.$ We have
\begin{align*}
F_{Z_{1},Z_{2}}(z_{1},z_{2})  &  =P\{Z_{1}\leq z_{1},Z_{2}\leq z_{2}\}\\
&  =P\{\varphi(Y)\leq z_{1},\psi(X)\leq z_{2}\}\\
&  =P\{Y\leq\varphi^{-1}(z_{1}),X\leq\psi^{-1}(z_{2})\}=P\{X\leq\psi
^{-1}(z_{2}),Y\leq\varphi^{-1}(z_{1})\}\\
&  =F_{X,Y}(\psi^{-1}(z_{2}),\varphi^{-1}(z_{1})),(z_{1},z_{2})\in\lbrack
a,b]^{2}%
\end{align*}
Hereafter, we assume that $\psi^{-1}(\infty)=\infty$ and $\varphi^{-1}%
(\infty)=\infty$ and $\psi^{-1}(-\infty)=-\infty$ and $\varphi^{-1}%
(-\infty)=-\infty.$ $\ $The marginal distributions are
\begin{align*}
F_{Z_{1}}(z_{1})  &  \equiv P\{Z_{1}\leq z_{1}\}=\lim_{z_{2}\rightarrow\infty
}P\{X\leq\psi^{-1}(z_{2}),Y\leq\varphi^{-1}(z_{1})\}=F_{Y}(\varphi^{-1}%
(z_{1}))\\
F_{Z_{2}}(z_{2})  &  \equiv P\{Z_{2}\leq z_{2}\}=\lim_{z_{1}\rightarrow\infty
}P\{X\leq\psi^{-1}(z_{2}),Y\leq\varphi^{-1}(z_{1})\}=F_{X}(\psi^{-1}(z_{2})).
\end{align*}
Now we are interested in the copula of random vector $(Z_{1},Z_{2})$. \ Denote%
\[
F_{X}^{-1}(x)=\inf\{y:F_{X}(y)\geq x\}\text{ and }F_{Y}^{-1}(y)=\inf
\{x:F_{Y}(x)\geq y\}.
\]
Now, \ consider
\begin{equation}
F_{Z_{1},Z_{2}}(z_{1},z_{2})=C_{Z_{1},Z_{2}}(F_{Z_{1}}(z_{1}),F_{Z_{2}}%
(z_{2})), \label{1}%
\end{equation}
where $C_{Z_{1},Z_{2}}(t,s)$ is a connecting copula of $Z_{1}$ and $Z_{2}.$

\bigskip

Using probability integral transformation
\begin{align}
F_{Z_{1}}(z_{1})  &  =t\text{ }\Leftrightarrow F_{Y}(\varphi^{-1}%
(z_{1}))=t\text{ }\Leftrightarrow z_{1}=\varphi(F_{Y}^{-1}(t)),\text{
}F_{Z_{1}}^{-1}(t)=\varphi(F_{Y}^{-1}(t))\nonumber\\
F_{Z_{2}}(z_{2})  &  =s\text{ }\Leftrightarrow F_{X}(\psi^{-1}(z_{2}))=s\text{
}\Leftrightarrow z_{2}=\psi(F_{X}^{-1}(s)),F_{Z_{2}}^{-1}(s)=\psi(F_{X}%
^{-1}(s)) \label{2}%
\end{align}
we obtain from (\ref{1}) and (\ref{2})
\begin{align*}
C_{Z_{1},Z_{2}}(t,s)  &  =F_{Z_{1},Z_{2}}(F_{Z_{1}}^{-1}(t),F_{Z_{2}}%
^{-1}(s))=F_{X,Y}(\psi^{-1}(z_{2}),\varphi^{-1}(z_{1}))\\
&  =F_{X,Y}(\psi^{-1}(\psi(F_{X}^{-1}(s)),\varphi^{-1}(\varphi(F_{Y}%
^{-1}(t))\\
&  =F_{X,Y}(F_{X}^{-1}(s),F_{Y}^{-1}(t))=C(s,t)
\end{align*}
Therefore, we can formulate the following theorem.

\begin{theorem}
\label{Copula Theorem 1}Let the joint distribution function of random
variables $X$ and $Y$ be $F_{X,Y}(x,y)=C(F_{X}(x),F_{Y}(y)),(x,y)\in\lbrack
a,b]^{2},$ where $C(u,v),(u,v)\in I^{2}\equiv\lbrack0,1]^{2}$ is a connected
copula. Consider functions $\varphi(y)=E(X\mid Y=y),y\in%
\mathbb{R}
$ and $\psi(x)=$ $E(Y\mid X=x),x\in%
\mathbb{R}
$ and random variables $Z_{1}\equiv\varphi(Y)=E(X\mid Y)$ and $Z_{2}\equiv
\psi(X)=E(Y\mid X).$ Assume that $\lim_{t\rightarrow\infty}\psi^{-1}%
(t)=\infty$ and $\lim_{s\rightarrow\infty}\varphi^{-1}(s)=\infty.$ Then the
copula of $Z_{1}$ and $Z_{2}$ is $C_{Z_{1},Z_{2}}(t,s)=C(s,t),$ $0\leq
t,s\leq1.$ Therefore, if $X$ and $Y$ are exchangeable then $C_{Z_{1},Z_{2}%
}(t,s)=C(t,s).$
\end{theorem}

\begin{example}
(\label{Normal})Let ($X;Y)$ be a bivariate normal random vector with joint
pdf
\begin{align*}
f(x,y)  &  =\frac{1}{2\pi\sigma_{1}\sigma_{2}\sqrt{1-\rho^{2}}}\exp\{-\frac
{1}{2(1-\rho^{2})}((\frac{x-\mu_{1}}{\sigma_{1}})^{2}\\
&  -2\rho(\frac{x-\mu_{1}}{\sigma_{1}})(\frac{x-\mu_{2}}{\sigma2}%
)+(\frac{y-\mu_{2}}{\sigma_{2}})^{2}\}.
\end{align*}
Then
\begin{align}
\psi(x)  &  =E(Y\mid X=x)=\mu_{2}+\rho\frac{\sigma_{2}}{\sigma_{1}}(x-\mu
_{1})\label{Cov1}\\
\varphi(y)  &  =E(X\mid Y=y)=\mu_{1}+\rho\frac{\sigma_{1}}{\sigma2}(y-\mu_{2})
\label{Cov2}%
\end{align}
and
\begin{align}
\psi^{-1}(t)  &  =\frac{t-\mu_{2}}{\rho\sigma_{2}}\sigma_{1}+\mu
_{1}\label{Cov3}\\
\varphi^{-1}(s)  &  =\frac{s-\mu_{1}}{\rho\sigma_{1}}\sigma_{2}+\mu_{2}
\label{Cov4}%
\end{align}
and
\begin{align*}
\lim_{t\rightarrow\infty}\psi^{-1}(t)  &  =\infty\\
\lim_{s\rightarrow\infty}\varphi^{-1}(s)  &  =\infty.
\end{align*}
Therefore, $C_{1}(x,y)=C(y,x),$ where $C$ is a copula of $(X$,$Y)$ and $C_{1}
$ is a copula of $(\psi(Y),\varphi(X))=(E(X\mid Y),E(Y\mid X)).$
\end{example}

The following well-known property of the conditional expectation state that
given $Y$ the conditional expectation of $XY$ is \
\begin{align*}
E(XY)  &  =E(E(XY\mid Y))=E(YE(X\mid Y))\\
E(XY)  &  =E(E(XY\mid X))=E(XE(Y\mid X))
\end{align*}

\begin{proposition}
\label{Covariance Proposition 1}%
\[
Cov(E(X\mid Y),Y)=Cov(E(Y\mid X),X)=Cov(X,Y).
\]

\end{proposition}

\begin{proof}
Since
\[
E(XY)=E[E(XY\mid Y)]=E[YE(X\mid Y)]
\]
and%
\[
E(Y)E[E(X\mid Y)]=E(Y)E(X)
\]
then,%

\begin{align*}
Cov(X,Y)  &  =E(XY)-E(X)E(Y)=E(YE(X\mid Y))-E(Y)E(E(X\mid Y))\\
&  =Cov(Y,E(X\mid Y))=Cov(E(X\mid Y),Y).
\end{align*}

\end{proof}

\bigskip

\begin{remark}
It can be observed that $Cov(E(X\mid Y),E(Y\mid X))$ may not be equal to
$Cov(X,Y).$
\end{remark}

Indeed,
\begin{align*}
&  E\psi(Y)\varphi(X)-E\psi(Y)E\varphi(X)\\
&  =E\psi(Y)\varphi(X)-E[E(X\mid Y)]E[E(Y\mid X)]\\
&  =E\psi(Y)\varphi(X)-E(X)E(Y)\\
&  =E[E(X\mid Y)(E(X\mid Y)]-E(X)E(Y).
\end{align*}
Let for example $\psi(Y)=aY+b,$ $\varphi(X)=cX+d,$ where $a,b,c,d>0$ (see for
example (\ref{Cov1}) and (\ref{Cov2})). Then $E\psi(Y)\varphi
(X)=acE(XY)+adEY+bcEX+bd.$ Therefore, $E\psi(Y)\varphi
(X)=acE(XY)+adEY+bcEX+bd=EXY$ only if $a=1,c=1,b=0,d=0.$ For (\ref{Cov1}) and
(\ref{Cov2}) this means that it must be $\mu_{1}=\mu_{2}=0,\sigma_{1}%
=\sigma_{2}=1,\rho=1.$

\section{Sequences of predicted random variables}

Let $X_{1},X_{2},...,X_{n},...$ be a sequence of dependent random variables.
$\ $Let $\ Y_{1}=X_{1},Y_{2}=E(X_{2}\mid X_{1}),...,Y_{n}=E(X_{n}\mid
X_{n-1}),...$ It is clear that $EY_{i}=E(E(X_{i}\mid X_{i-1}))=EX_{i}%
,i=12,...$ Since $E(Y_{1}Y_{2})=E(X_{1}E(X_{2}\mid X_{1}))=E(E(X_{1}X_{2}\mid
X_{1}))=EX_{1}X_{2},$ then $Cov(Y_{1},Y_{2})=Cov(X_{1},X_{2}).$ \ Furthermore,
denoting by $\psi_{i}(x)=E(X_{i}\mid X_{i-1}),$ $i=2,3,...,$ we have

\bigskip It is well known that the best predictor for $X_{n+1}$ expressed as a
function of $X_{1},X_{2},...,X_{n}$ is \ $E(X_{n+1}\mid X_{1},X_{2}%
,...,X_{n})=\Psi(X_{1},X_{2},...,X_{n}),$ i.e.
\[
\underset{g}{min}E(X_{n+1}-g(X_{1},X_{2},...,X_{n}))^{2}=E(X_{n+1}%
-E(X_{n+1}\mid X_{1},X_{2},...,X_{n}))^{2}%
\]

\begin{theorem}
\bigskip\label{Theorem 3AA} Let $X,Y$ and $Z$ be any random variables defined
on probability space $\left\{  \Omega,\digamma,P\right\}  .$ Then
\begin{equation}
E[X-E(X\mid Y,Z)]^{2}\leq\min\left\{  E[X-E(X\mid Y)]^{2},E[X-E(X\mid
Z)]^{2}\right\}  \label{P0}%
\end{equation}

\end{theorem}

\begin{proof}
Consider%
\begin{align*}
&  E\left[  (X-E(X\mid Y,Z))^{2}\mid Y,Z\right] \\
&  =E[(X-E(X\mid Y)+E(X\mid Y)-E(X\mid Y,Z))^{2}\mid Y,Z]
\end{align*}%
\begin{align}
&  =E[\{X-E(X\mid Y)\}^{2}\mid Y,Z]\nonumber\\
&  +2E\left[  \{X-E(X\mid Y)\}\{E(X\mid Y)-E(X\mid Y,Z)\}\mid Y,Z\right]
\nonumber\\
&  +E\left[  \{E(X\mid Y)-E(X\mid Y,Z)\}^{2}\mid Y,Z\right] \nonumber\\
&  =E[\{X-E(X\mid Y)\}^{2}\mid Y,Z]\label{P1}\\
&  +2\left[  E(X\mid Y)-E(X\mid Y,Z)\right]  E\left[  \left\{  X-E(X\mid
Y)\right\}  \mid Y,Z\right] \nonumber\\
&  +E\left[  \left\{  E(X\mid Y)-E(X\mid Y,Z)\right\}  ^{2}\mid Y,Z\right]
.\nonumber
\end{align}
In (\ref{P1}) we take into account the fact that $h(Y,Z)\equiv E(X\mid
Y)-E(X\mid Y,Z)$ is $Y,Z$ measurable and behaviors as a constant in
conditional expectation with respect to $Y,Z$ and, therefore
\begin{align*}
&  E\left[  \left\{  X-E(X\mid Y)\right\}  \left\{  E(X\mid Y)-E(X\mid
Y,Z)\right\}  \mid Y,Z\right] \\
&  =E\left[  \{X-E(X\mid Y)\}h(Y,Z)\mid Y,Z\right] \\
&  =h(Y,Z)E\left[  \left\{  X-E(X\mid Y)\right\}  h(Y,Z)\mid Y,Z\right] \\
&  =\left[  E(X\mid Y)-E(X\mid Y,Z)\right]  E\left[  X-E(X\mid Y)\}\mid
Y,Z\right]  .
\end{align*}
Since
\begin{align}
E(E[\{X-E(X  &  \mid Y)\}\mid Y,Z]\nonumber\\
&  =E(E(X\mid Y,Z))-E(E(X\mid Y)\mid Y,Z)\nonumber\\
&  =EX-E(E(X\mid Y))=EX-EX=0. \label{P2}%
\end{align}
Therefore,
\begin{align}
E[X-E(X  &  \mid Y,Z)]^{2}\nonumber\\
&  =E[\{X-E(X\mid Y)\}^{2}\mid Y,Z]\nonumber\\
+E[\{E(X  &  \mid Y)-E(X\mid Y,Z)\}^{2}\mid Y,Z] \label{pp2}%
\end{align}
Applying the operatior $E$ to both sides of (\ref{pp2}) we obtain%
\begin{align*}
E[X-E(X  &  \mid Y,Z)]^{2}\\
&  =E[\{X-E(X\mid Y)\}^{2}]\\
+[E(X  &  \mid Y)-E(X\mid Y,Z)]^{2}%
\end{align*}
which imply
\begin{align*}
E[X-E(X  &  \mid Y,Z)]^{2}\\
&  \leq E[\{X-E(X\mid Y)\}^{2}].
\end{align*}

\end{proof}

\bigskip

\begin{corollary}
\bigskip\label{Corollay 1}For any $n\geq2,$ and the sequence of random
variables $X_{1},X_{2},...,X_{n},...$ it is true that
\begin{align*}
E\left[  X_{n+1}-E(X_{n+1}\mid X_{i_{1}},X_{i_{2}},...,X_{i_{l}})\right]
^{2}  &  \leq E\left[  X_{n+1}-E(X_{n+1}\mid X_{i_{1}},X_{i_{2}},...,X_{i_{k}%
})\right]  ^{2},\\
1  &  \leq i_{1}<i_{2}<...<i_{k}<i_{l}\leq n,1\leq k<l\leq n.
\end{align*}
For example,
\begin{align*}
E\left[  X_{n+1}-E(X_{n+1}\mid X_{1},X_{2},...,X_{n})\right]  ^{2}  &  \leq
E\left[  X_{n+1}-E(X_{n+1}\mid X_{1},X_{2},...,X_{n-1})\right]  ^{2}\\
E\left[  X_{n+1}-E(X_{n+1}\mid X_{1},X_{2},...,X_{n})\right]  ^{2}  &  \leq
E\left[  X_{n+1}-E(X_{n+1}\mid X_{1},X_{2},...,X_{n-2})\right]  ^{2}\\
E\left[  X_{n+1}-E(X_{n+1}\mid X_{1},X_{2},...,X_{n})\right]  ^{2}  &  \leq
E\left[  X_{n+1}-E(X_{n+1}\mid X_{2},X_{3},...,X_{n})\right]  ^{2}\\
E\left[  X_{n+1}-E(X_{n+1}\mid X_{2},X_{3},X_{4}\right]  ^{2}  &  \leq
E\left[  X_{n+1}-E(X_{n+1}\mid X_{2},X_{3})\right]  ^{2}%
\end{align*}
etc.
\end{corollary}

\begin{example}
\bigskip\textbf{(Martingale)} The sequence $X_{1},X_{2},...,X_{n},...$ is
called a martingale if
\[
E(X_{n+1}\mid X_{1},X_{2},...,X_{n})=X_{n}.
\]
It follows from the corollary that if $X_{1},X_{2},...,X_{n},...$ is a
martingale then
\begin{align*}
&  E[X_{n+1}-X_{n}]^{2}\\
&  =E\left[  X_{n+1}-E(X_{n+1}\mid X_{1},X_{2},...,X_{n})\right]  ^{2}\\
&  \leq E\left[  X_{n+1}-E(X_{n+1}\mid X_{i_{1}},X_{i_{2}},...,X_{i_{k}%
})\right]  ^{2}\\
1  &  \leq i_{1}<i_{2}<...<i_{k}\leq n,1\leq k\leq n.
\end{align*}

\end{example}

\begin{example}
\textbf{(Markov chain)} Let $X_{1},X_{2},...,X_{n},...$ be a Markov chain,
i.e. for an interval $A$ of the real line it is true that
\[
P\{X_{n}\in A\mid X_{i_{1}},X_{i_{2}},...,X_{i_{k}}\}=P\{X_{n}\in A\mid
X_{i_{k}}\}
\]%
\begin{align*}
E\left[  X_{n}-E(X_{n}\mid X_{i_{l}})\right]  ^{2}  &  =E\left[  X_{n}%
-E(X_{n}\mid X_{i_{1}},X_{i_{2}},...,X_{i_{l}})\right]  ^{2}\\
&  \leq E\left[  X_{n}-E(X_{n}\mid X_{i_{1}},X_{i_{2}},...,X_{i_{k}})\right]
^{2}\\
&  =E\left[  X_{n}-E(X_{n+1}\mid X_{i_{k}})\right]  ^{2},\\
1  &  \leq i_{1}<i_{2}<...<i_{k}<i_{l}\leq n,1\leq k<l\leq n.
\end{align*}
Therefore,,
\begin{align*}
E\left[  X_{n+1}-E(X_{n+1}\mid X_{l})\right]  ^{2}  &  \leq E\left[
X_{n+1}-E(X_{n+1}\mid X_{k})\right]  ^{2},\\
1  &  \leq k<l\leq n.
\end{align*}
A good illustration of this fact can be given with order statistics. Let
$X_{1:n}\leq X_{2:n}\leq\cdots\leq X_{n:n}$ be the order statistics of
$X_{1},X_{2},...,X_{n}.$ Then for $1\leq k<l\leq n$ one can write
\begin{equation}
E\left[  X_{n:n}-E(X_{n:n}\mid X_{l:n})\right]  ^{2}\leq E\left[
X_{n:n}-E(X_{n:n}\mid X_{k:n})\right]  ^{2}, \label{e2}%
\end{equation}
i.e. $E(X_{n:n}\mid X_{l:n})$ \ predicts $X_{n:n}$ better than $E(X_{n:n}\mid
X_{k:n}).$
\end{example}

\begin{example}
\textbf{(Order statistics) }Let $X_{1},X_{2},...,X_{n}$ be iid random
variables and $X_{1:n}\leq X_{2:n}\leq\cdots\leq X_{n:n}$ be the order
statistics. It is well known that the order statistics form a Markov chain.
\ Therefore, $E(X_{n:n}\mid X_{1:n,}...,X_{n-1:n},X_{n:n})=E(X_{n:n}\mid
X_{n-1:n}).$ \ By Markov property,
\begin{align*}
&  E\left[  X_{n:n}-E(X_{n:n}\mid X_{1:n},X_{2:n},...,X_{n-1:n})\right]
^{2}\\
&  =E\left[  X_{n:n}-E(X_{n:n}\mid X_{n-1:n})\right]  ^{2}.
\end{align*}
from the corollary and (\ref{e2}) we have
\[
E[X_{n:n}-E(X_{n:n}\mid X_{n-1:n})]^{2}\leq E[X_{n:n}-E(X_{n:n}\mid
X_{n-2:n})]^{2}.
\]
It means that $E(X_{n:n}\mid X_{n-1:n})$ is better estimation for $X_{n:n}$
than $E(X_{n:n}\mid X_{n-2:n}).$ Let us compute the functions $g_{1}%
(x)=E(X_{n:n}\mid X_{n-1:n}=x)$ and $g_{2}(x)=E(X_{n:n}\mid X_{n-2:n}=x).$
From the joint distribution of $X_{r:n}$ and $X_{s:n},$ $r<s,$ we can easily
write the conditional pdf's of $X_{n.n}\mid X_{n-1:n}$ and $X_{n.n}\mid
X_{n-2:n}$ as
\begin{align*}
f_{n\mid n-1}\left(  z\mid x\right)   &  =\frac{f(z)}{1-F(x)},x<z\\
f_{n\mid n-2}\left(  z\mid x\right)   &  =\frac{2(F(z)-F(x))}{(1-F(x))^{2}%
}f(z),x<z,
\end{align*}
respectively. Then for a uniform(0,1) distribution, we can write
\begin{align*}
g_{1}(x)  &  =E(X_{n:n}\mid X_{n-1:n}=x)=\int\limits_{x}^{1}f_{n\mid
n-1}(z\mid x)dz\\
&  =\frac{1+x}{2},0\leq x\leq1
\end{align*}
and%
\begin{align*}
g_{2}(x)  &  =E(X_{n:n}\mid X_{n-2:n}=x)=2\int\limits_{x}^{1}f_{n\mid
n-2}(z\mid x)dz\\
&  =\frac{x+2}{3},0\leq x\leq1
\end{align*}
and it is clear that
\[
g_{1}(x)<g_{2}(x),0\leq x\leq1
\]
because
\[
\frac{x+2}{3}=\frac{1+x}{2}+\frac{1-x}{6},0\leq x\leq1.
\]
This means that $\ g_{1}(X_{n-1:n})=E(X_{n:n}\mid X_{n-1:n})>$ $\ g_{2}%
(X_{n-2:n})=E(X_{n:n}\mid X_{n-2:n}),$ hence $E(X_{n:n}\mid X_{n-1:n})$ is
better than $E(X_{n:n}\mid X_{n-2:n})$ as a predictor of $X_{n:n}.$
\bigskip\ For an exponential distribution $F(x)=1-\exp(-x),x\geq0$ it can be
easily verify that
\begin{align*}
g_{1}(x)  &  =E(X_{n:n}\mid X_{n-1:n}=x)\\
&  =\frac{1}{1-F(x)}\int\limits_{x}^{\infty}zf_{n\mid n-1}(z\mid x)dz\\
&  =\frac{1}{1-F(x)}\int\limits_{x}^{\infty}zf(z)dz\\
&  =e^{x}\int\limits_{x}^{\infty}ze^{-z}dz=e^{x}e^{-x}(x+1)=x+1
\end{align*}
i.e. $g_{1}(x)>g_{2}(x),x\geq0$ and again $E(X_{n:n}\mid X_{n-1:n})$ is better
than $E(X_{n:n}\mid X_{n-2:n})$ as a predictor of $X_{n:n}.$
\end{example}

\begin{example}
(Record Values) Let $X_{1},X_{2},...,X_{n},...$ be a sequence of independent
identically distributed (i.i.d.) r.v.'s with continuous d.f. $F$ ;
$X_{1:n}\leq X_{2:n}\leq...\leq X_{n:n}$ be the order statistics of
$X_{1},X_{2},...,X_{n}$. The random variable $X_{K}$ is called a (upper)
record value of the sequence $\left\{  X_{n},n\geq1\right\}  $ if $X_{K}%
>\max\left\{  X_{1},X_{2},...,X_{K-1}\right\}  .$ By convention $X_{1}$ is
record value. \ Denote by $\left\{  U(n),n>1\right\}  $ \ \ the sequence of
record times:
\[
U(n)=\min\left\{  j:j>U(n-1),X_{j}>\;X_{U(n-1)}\right\}  ,n>1\text{ \ \ \ with
}U(1)=1.
\]
$X_{U(n)}$ is called $n$ th upper record value. Developments on records have
been reviewed by many authors including Nevzorov (1988), Nagaraja (1988),
Arnold and Balakrishnan (1989), Arnold, Balakrishnan, Nagaraja (1992),
Ahsanullah (1995). The properties of records values of iid random variables
have been extensively studied in the literature. Many properties of records
can be expressed in terms of the functions $R\left(  x\right)  =-\log\bar
{F}\left(  x\right)  $ where $\bar{F}\left(  x\right)  =1-F\left(  x\right)  $
and $0<\bar{F}\left(  x\right)  <1$. It is well known that, the sequence of
record values $X_{U(1)},X_{U(2)},...,X_{U(n)},...$ form a Markov chain. \ From
the corollary and (\ref{e2}) we have
\begin{align}
&  E\left[  X_{U(n)}-E(X_{U(n)}\mid X_{U(1)},X_{U(2)},...,X_{U(n-1)})\right]
^{2}\nonumber\\
&  \leq E\left[  X_{U(n)}-E(X_{U(n)}\mid X_{U(1)},X_{U(2)},...,X_{U(n-2)}%
)\right]  ^{2} \label{R1}%
\end{align}
By Markov property
\begin{align}
E\left(  X_{U(n)}\mid X_{U(1)},X_{U(2)},...,X_{U(n-1)}\right)   &  =E\left(
X_{U(n)}\mid X_{U(n-1)}\right) \nonumber\\
E\left(  X_{U(n)}\mid X_{U(1)},X_{U(2)},...,X_{U(n-2)}\right)   &  =E\left(
X_{U(n)}\mid X_{U(n-2)}\right)  . \label{R2}%
\end{align}
Then from (\ref{R1}) and (\ref{R2}) for any $n>2$ we have
\begin{equation}
E\left[  X_{U(n)}-E(X_{U(n)}\mid X_{U(n-1)})\right]  ^{2}\leq E\left[
X_{U(n)}-E(X_{U(n)}\mid X_{U(n-2)})\right]  ^{2}, \label{R3}%
\end{equation}
i.e. $E(X_{U(n)}\mid X_{U(n-1)})$ is better than $E(X_{U(n)}\mid X_{U(n-2)})$
as a predictor of $X_{U(n)}.$ It is clear that, (\ref{R3}) can be extended as
\begin{align}
E\left[  X_{U(n)}-E(X_{U(n)}\mid X_{U(l)})\right]  ^{2}  &  \leq E\left[
X_{U(n)}-E(X_{U(n)}\mid X_{U(k)})\right]  ^{2},\label{R4}\\
2  &  <k<l<n.\nonumber
\end{align}

\end{example}

\bigskip It is possible to extend the list of examples to the areas where the
prediction of random variables with conditional expectations is the subject.

\begin{conclusion}
This paper investigates the new properties of conditional expectation with
respect to a sigma-algebra generated by other random variables. The
conditional expectations of the random variable with respect to a
sigma-algebra generated by the random variable is its best predictor in the
sense of least square distance. Some important inequalities concerning the
predictions of random variables are proved. These inequalities can find
important applications in many areas such as financial mathematics, actuarial
sciences, and reliability engineering. An application of the main inequality
having interesting consequences in per-share stock is presented. Considering
conditional expectations as random variables, we study also the dependence
properties of and copulas of these random variables. Some examples with
ordered random variables and martingales are provided.
\end{conclusion}

\bigskip

\bigskip

\end{document}